\documentclass{birkjour}

\usepackage{amsfonts, amsmath, amssymb,latexsym}
\usepackage{url}

\newcommand{\be}[1]{\operatorname{e}_{\{#1\}}}
\newcommand{\mb}[1]{\mathbb{#1}}

\newcommand{\C}{\mb{C}}
\newcommand{\QQ}{\mb{Q}}
\newcommand{\R}{\mb{R}}
\newcommand{\Real}{\R}
\newcommand{\Complex}{\mathbb C}
\newcommand{\Quat}{\mathbb H}

\newtheorem{Lemma}{Lemma}

\newtheorem{Theorem}{Theorem}

\newtheorem{Proposition}{Proposition}

\newtheorem{Corollary}[Lemma]{Corollary}

\newtheorem{Remark}{Remark}

\newtheorem{Definition}{Definition}

\begin{document}

\title[Quasi-Clifford algebras and Hadamard matrices]
{Gasti\-neau-Hills' quasi-Clifford algebras \\ and plug-in constructions for Hada\-mard \\ matrices}

\author[Paul Leopardi]{Paul~C.~Leopardi}%
\address{%
Australian Government - Bureau of Meteorology.\\
The University of Melbourne.}
\email{paul.leopardi@gmail.com}

\begin{abstract}
The quasi-Clifford algebras as described by Gasti\-neau-Hills in 1980 and 1982, should be better known,
and have only recently been rediscovered.
These algebras and their representation theory provide effective tools to address the following problem
arising from a plug-in construction for Hada\-mard matrices:
Given $\lambda$, a pattern of amicability / anti-amicability, with $\lambda_{j,k}=\lambda_{k,j}=\pm 1$,
find a set of $n$ monomial $\{-1,0,1\}$ matrices $D$ of minimal order such that
\begin{align*}
D_j D_k^T - \lambda_{j,k} D_k D_j^T &= 0 \quad (j \neq k).
\end{align*}
\end{abstract}
\maketitle

\section{Introduction}
\label{sec-Introduction}

The work of Gasti\-neau-Hills on quasi-Clifford algebras \cite{Gas80,Gas82} should be better known.
In particular, as at April 2018, his paper on the subject \cite{Gas82} had only four citations
other than self-citations, according to Google, \cite{DelF11,Leo14Constructions,Mar2018classification,Seb17gastineau},
and only one of these \cite{Mar2018classification} discusses quasi-Clifford algebras and their representation theory to any depth.

Since the quasi-Clifford algebras are a fundamental and natural generalization of the Clifford algebras,
these algebras, or some subset of them, as well as their representation theory, have been rediscovered
or partially rediscovered a number of times.
The rediscoveries include
\begin{itemize}
\item
da Rocha and Vaz' extended Clifford algebras \cite{daR06} which are doubled real Clifford algebras,
including the simplest case of a quasi-Clifford algebra that is not itself a Clifford algebra;
\item
Rajan and Rajan's extended Clifford algebras \cite{Raj07algebraic,Raj07stbcs},
which are a subset of the special quasi-Clifford algebras; and
\item
Marchuk's extended Clifford algebras \cite{Mar2018classification},
which correspond to the special quasi-Clifford algebras over the field of real numbers.
\end{itemize}
The partial rediscoveries include near misses,
such as the rediscoveries of the finite groups generated by the basis elements
of the quasi-Clifford algebras, and their representations:
\begin{itemize}
\item
The finite groups generated by the basis elements of the quasi-Clifford algebras are signed groups
and have real monomial representations of the order of a power of two \cite{Cra95}.
\item
The non-abelian extensions of $C_2$ by $C_2^k$ as classified by de Launey and Flannery's Theorem 21.2.3
\cite[Section 21.2]{DelF11} and the finite groups generated by the basis elements of the quasi-Clifford algebras
are connected by their relationship to the finite groups generated by the basis elements of Clifford algebras.
The correspondence between the two deserves further investigation.
See in particular, Lam and Smith's classification of the finite groups generated by the basis elements of Clifford algebras \cite{LamS89}.
See also the classifications given by de Launey and Smith \cite{DelS01}.
\end{itemize}

The original application of quasi-Clifford algebras and their representation theory
was to systems of orthogonal designs \cite{Gas80,Gas82,Seb17gastineau}.
The current paper applies quasi-Clifford algebras and their representation theory
to the study of some plug-in constructions for Hada\-mard matrices described by the author in 2014 \cite{Leo14Constructions}.
The key question addressed is:
Given $\lambda$, a pattern of amicability / anti-amicability, with $\lambda_{j,k}=\lambda_{k,j}=\pm 1$,
find a set of $n$ monomial $\{-1,0,1\}$ matrices $D$ of minimal order such that
\begin{align*}
D_j D_k^T - \lambda_{j,k} D_k D_j^T &= 0 \quad (j \neq k).
\end{align*}
Specifically, this paper contains a new proof of Theorem 5 of \cite{Leo14Constructions} that
answers Question 1 of that paper.

The remainder of the paper is organized as follows.
Section \ref{sec-Quasi-Clifford} outlines Gasti\-neau-Hills' theory of quasi-Clifford algebras.
Section \ref{sec-Plug-in} revises the plug-in constructions for Hada\-mard matrices.
Section \ref{sec-SQC-Plug-in} uses the theory of real Special quasi-Clifford algebras to address questions related to those constructions.

\section{Quasi-Clifford algebras}\label{sec-Quasi-Clifford}

Humphrey Gasti\-neau-Hills fully developed the theory of quasi-Clifford algebras in his thesis of 1980
\cite{Gas80} and published the key results in a subsequent paper \cite{Gas82}.
The paper describes the theory of quasi-Clifford algebras in full generality for fields of
characteristic other than 2. This paper uses only the properties of quasi-Clifford algebras
over the real field, and after giving the general definition, this section presents a summary of
Gasti\-neau-Hills' constructions and results in this case.
\newpage
\begin{Definition}
\label{Gas82-2-1}
\cite[(2.1)]{Gas82}
Let $F$ be a commutative field of characteristic not 2, $m$ a
positive integer, $(\kappa_i),\ 1 \leqslant i \leqslant m$
a family of non-zero elements of $F$, and $(\delta_{i,j}),\ 1 \leqslant i\leqslant j \leqslant m$ a
family of elements from $\{0, 1\}$.
The quasi-Clifford, or QC, algebra $\mathcal{C} = \mathcal{C}_F [m, (\kappa_i), (\delta_{i,j})]$
is the algebra (associative, with a 1) over $F$ on $m$ generators
$\alpha_1,\ldots, \alpha_m,$ with defining relations
\begin{align}
\alpha_i^2 = \kappa_i, \quad \alpha_j \alpha_i = (-1)^{\delta_{i,j}} \alpha_i \alpha_j \quad (i < j)
\end{align}
(where $\kappa_i$ of $F$ is identified with $\kappa_i$ times the $1$ of $\mathcal{C}$).
\end{Definition}

If all $\delta_{i,j} = 1$ we have a Clifford algebra corresponding to some non-singular
quadratic form on $F^m$ \cite{Lam73}. If in addition each $\kappa_i = \pm 1$ we have the
Special Clifford algebras studied by Kawada and Imahori \cite{Kaw50structure},
Porteous \cite{Por69,Por95} and Lounesto \cite{Lou97} amongst others.

\begin{Theorem}
\label{Gas82-2-3}
\cite[(2.3)]{Gas82}
The QC algebra $\mathcal{C}$ of Definition \ref{Gas82-2-1} has dimension $2^m$ as a vector space over $F$,
and a basis is $\{\alpha_1^{\epsilon_1}\ldots \alpha_m^{\epsilon_m}, \epsilon_i = 0\ \text{or}\ 1\}$.
\end{Theorem}

This paper concentrates on the QC algebras for which each $\kappa_i = \pm 1$.
Gasti\-neau-Hills call such algebras Special quasi-Clifford, or SQC, algebras.

Also, from this point onwards, the field $F$ is the real field $\Real$,
and Gasti\-neau-Hills' key theorems and constructions are summarised for this case.
Additionally, the real Special Clifford algebras are referred to simply as Clifford algebras,
and the notation of Porteous \cite{Por69,Por95} and Lounesto \cite{Lou97}
is used for these algebras and their representations.

Gasti\-neau-Hills \cite{Gas82}
uses the notation $[\alpha_1,\ldots,\alpha_m]$ for the QC algebra
generated by $\alpha_1,\ldots,\alpha_m$, and the following notation for two special cases.
In the case of a single generator, $\C_b := [\beta]$ where $\beta^2=b$.
For a pair of anti-commuting generators, $\QQ_{c,d} := [\gamma,\delta]$ where $\gamma^2=g$, $\delta^2=d$.
This notation yields the following isomorphisms between these low dimensional real SQC algebras and their corresponding Clifford algebras \cite[(2.2)]{Gas82}.

\begin{align}
\C_{-1} &\simeq \R_{0,1} \simeq \C,
\notag
\\
\C_1 &\simeq \R_{1,0} \simeq {}^2\R,
\label{Gas82-2-2}
\\
\QQ_{-1,-1} &\simeq \R_{0,2} \simeq \Quat,
\notag
\\
\QQ_{-1,1} &\simeq \R_{1,1} \simeq \R(2),
\notag
\\
\QQ_{1,-1} &\simeq \R_{1,1} \simeq \R(2),
\notag
\\
\QQ_{1,1} &\simeq \R_{2,0} \simeq \R(2).
\notag
\end{align}

Gasti\-neau-Hills first decomposition theorem in the special case of real SQC algebras is as follows.
\begin{Theorem}
\label{Gas82-2-7}
\cite[(2.7)]{Gas82}
Any real SQC algebra $\mathcal{C}[m, (\kappa_i), (\delta_{i,j})] = [\alpha_1,\ldots,\alpha_m]$ is expressible
as a tensor product over $\R$:
\begin{align}
\mathcal{C}
&\simeq \C_{b_1} \otimes \ldots \otimes \C_{b_r} \otimes \QQ_{c_1,d_1} \otimes \ldots \otimes \QQ_{c_s,d_s}
\notag
\\
&= [\beta_1]  \otimes \ldots \otimes [\beta_r] \otimes [\gamma_1,\delta_1] \otimes \ldots \otimes [\gamma_s,\delta_s]
\label{eq-Gas82-2-7-i}
\end{align}
where $r, s \geqslant 0,$ $r + 2 s = m,$ and each $b_i, c_j, d_k$ is $\pm 1.$

Each $\beta_i, \gamma_j, \delta_k$
(where $\beta_i^2=b_i, \gamma_j^2=c_j, \delta_k^2=d_k$ and all pairs
commute except $\delta_i \gamma_i = -\gamma_i \delta_i, 1 \leqslant i \leqslant s$)
is, to within multiplication by $\pm 1$,
one of the basis elements $\alpha_1^{\epsilon_1}\ldots\alpha_m^{\epsilon_m}$ of $\mathcal{C}$.
Conversely each $\alpha_1^{\epsilon_1}\ldots\alpha_m^{\epsilon_m}$ is,
to within division by $\pm 1$,
one of
\begin{align*}
\beta_1^{\theta_1}\ldots\beta_r^{\theta_r}\gamma_1^{\phi_1}\delta_1^{\psi_1}\ldots\gamma_s^{\phi_s}\delta_s^{\psi_s}
\end{align*}
(each $\theta_i$, $\phi_j$, $\psi_k$ = 0 or 1).
Thus the latter
$2^{r+2s} = 2^m$ elements form a new basis of $\mathcal{C}$, and $\{\beta_i, \gamma_j, \delta_k\}$ is a new set of
generators.
\end{Theorem}

Here the tensor product $\otimes$ is the real tensor product of real algebras, as per Porteous \cite{Por69}.

Gastineau-Hills \cite{Gas82} goes on to investigate the Wedderburn structure of the real SQC algebras by first determining the centre of
each algebra, and then determining the irreducible representations.

\begin{Lemma}
\label{Gas82-2-8}
\cite[(2.8)]{Gas82}
The centre of $\mathcal{C} = [\beta_1] \otimes \ldots \otimes [\beta_r] \otimes [\gamma_1,\delta_1] \otimes \ldots \otimes [\gamma_s,\delta_s]$
($\beta_i, \gamma_j, \delta_k$ as in Theorem \ref{Gas82-2-7}) is the
$2^r$-dimensional subalgebra $[\beta_1]  \otimes \ldots \otimes [\beta_r]$.
\end{Lemma}

\begin{Remark}
\label{Gas82-2-9}
\cite[(2.9)]{Gas82}
The converse of Theorem \ref{Gas82-2-7} is obviously also true:
any algebra of the form \eqref{eq-Gas82-2-7-i} is a QC algebra.
Indeed, regarded as an algebra
on the generators $\{\beta_i, \gamma_j, \delta_k\}$, $\mathcal{C}$ of the form \eqref{eq-Gas82-2-7-i} is the QC algebra
$\mathcal{C}[r + 2s, (\kappa_i), (\delta_{i,j})]$
where $\kappa_1,\ldots, \kappa_{r+2s} = b_1,\ldots, b_r, c_1, d_1, \ldots c_s, d_s,$
respectively, and all $\delta_{i,j} = 0$ except $\delta_{r+2i-1, r+2i} = 1$ for $1 \leqslant i \leqslant s.$
\end{Remark}

\begin{Theorem}
\label{Gas82-2-10}
\cite[(2.10)]{Gas82}
The class of SQC algebras over $\R$ is the smallest class which is
closed under tensor products over $\R$ and which contains the Clifford algebras.
It is the smallest class
which is closed under tensor products over $\R$ and contains the algebras $\C_b$, $\QQ_{c,d}$
($b, c, d = \pm 1$).
The Clifford algebras are the QC algebras with 1- or
2-dimensional centres (general QC algebras can have $2^r$-dimensional centres, $r$ any
non-negative integer).
\end{Theorem}

\begin{Theorem}
\label{Gas82-2-11}
\cite[(2.11)]{Gas82}
Every real SQC algebra $\mathcal{C}[m, (\kappa_i), (\delta_{i,j})]$ is semi-simple.
\end{Theorem}

\begin{Remark}
\label{Gas82-3-2}
\cite[(3.2)]{Gas82}
There are irreducible representations of
$\C_b,$ $\QQ_{c,d}$ $(b, c, d = \pm 1)$ in which $\beta,\gamma,\delta$ are each represented by monomial
$\{-1,0,1\}$ matrices.
\end{Remark}

\begin{Remark}
\label{Gas82-3-3}
\cite[(3.3)]{Gas82}
Following from \eqref{Gas82-2-2}
the decomposition of a real SQC algebra takes
(possibly after reordering the factors) the form:
\begin{align}
\mathcal{C}
&= [\alpha_1,\ldots,\alpha_m]
\notag
\\
&\simeq
{}^2\R \otimes \ldots \otimes {}^2\R \otimes
\C \otimes \ldots \otimes \C \otimes
\Quat \otimes \ldots \otimes \Quat \otimes
\R(2) \otimes \ldots \otimes\R(2)
\label{eq-Gas-3-3}
\\
&\simeq
[\beta_1]  \otimes \ldots \otimes [\beta_r] \otimes
[\gamma_1,\delta_1] \otimes \ldots \otimes [\gamma_s,\delta_s]
\notag
\end{align}
where each $\beta_i, \gamma_j, \delta_k$ is plus or minus a product of the $\alpha_i$, and conversely each $\alpha_i$ is
plus or minus a product of the $\beta_i, \gamma_j, \delta_k$.
In general, each of ${}^2\R$, $\C$, $\Quat$, $\R(2)$ may appear zero or more times in the tensor product \eqref{eq-Gas-3-3}.
\end{Remark}

We now come to a well known lemma used in the representation theory of real and complex Clifford algebras.
\begin{Lemma}
\label{Gas82-3-4}
\cite[(3.4)]{Gas82} \cite[Prop. 10.44]{Por69} \cite[Prop. 11.9]{Por95}

(i) $\C \otimes \C \simeq {}^2\R \otimes \C \simeq {}^2\C.$

(ii) $\C \otimes \Quat \simeq \C \otimes \R(2) \simeq \C(2).$

(iii) $\Quat \otimes \Quat \simeq \R(2) \otimes \R(2) \simeq \R(4).$
\end{Lemma}

Remark \ref{Gas82-3-3} and the repeated application of Lemma \ref{Gas82-3-4} lead to the following result.
\begin{Theorem}
\label{Gas82-3-7}
\cite[(3.7)]{Gas82}

The Wedderburn structure of a real SQC algebra
$\mathcal{C}[m, (\kappa_i), (\delta_{i,j})]$
as a direct sum of full matrix algebras over division
algebras is (depending on $m, (\kappa_i), (\delta_{i,j})$) one of

(i) ${}^{2^r}\Real(2^s)$,

(ii) ${}^{2^{r-1}}\C \otimes \Real(2^s)$, or

(iii) ${}^{2^r}\Quat \otimes \Real(2^{s-1})$,

where in each case $r + 2 s = m,$ and $2^r$ is the dimension of the centre.
Conversely (as in Remark \ref{Gas82-2-9}) any such algebra (i), (ii) or (iii) is an SQC algebra
$\mathcal{C}[r + 2 s, (\kappa_i), (\delta_{i,j})]$
with respect to certain generators.
Also (as in Theorem \ref{Gas82-2-10}) the subclass of
algebras with structures (i), (ii) or (iii) for which $r \leqslant 1$ is precisely the class of
algebras isomorphic to Clifford algebras on $r + 2 s$ generators.
\end{Theorem}

\begin{Corollary}
\label{Gas82-3-8}
\cite[(3.8)]{Gas82}
In case (i) of Theorem \ref{Gas82-3-7} there are $2^r$ inequivalent irreducible representations, of order $2^s$;
in case (ii) $2^{r-1}$ of order $2^{s+1}$, and
in case (iii) $2^r$ of order $2^{s+1}.$
Any representation must be of order a multiple of (i) $2^s$, (ii) $2^{s+1}$, (iii) $2^{s+1}$ respectively.
\end{Corollary}

As a result of the well-known constructions that lead to Remark \ref{Gas82-3-2},
Gasti\-neau-Hills establishes the following result.
\begin{Theorem}
\label{Gas82-3-10}
\cite[(3.10)]{Gas82}
Each representation of a real SQC algebra
\newline
$\mathcal{C}[m, (\kappa_i), (\delta_{i,j})]$ on generators
$(\alpha_i)$ is equivalent to a matrix representation in which each $\alpha_i$ corresponds to a monomial
$\{-1,0,1\}$ matrix, which is therefore orthogonal.
\end{Theorem}

\section{Plug-in constructions for Hada\-mard matrices}\label{sec-Plug-in}

A recent paper of the author \cite{Leo14Constructions} describes a generalization of
Williamson's construction for Hada\-mard matrices \cite{Wil44}
using the real monomial representation of the basis elements of the Clifford algebra $\R_{m,m}$.
(Recall that $\R_{p,q}$ is the real universal Clifford algebra of the $2^{p+q}$ dimensional real quadratic space $\R^{p,q}$,
with $p+q$ anticommuting generators, $\be{-q}, \ldots, \be{-1}, \be{1}, \ldots \be{p}$ with $\be{k}^2 = -1$ if $k < 0$, $\be{k}^2 = 1$ if $k > 0$,
and that $\R_{m,m} \simeq \R(2^m)$, the algebra of real matrices with $2^m$ rows and $2^m$ columns \cite{Leo14Constructions,Lou97,Por69}.)

Briefly, the general construction uses some
\begin{align*}
\quad &A_k \in \{-1,0,1\}^{n \times n}, \quad B_k \in \{-1,1\}^{b \times b},
\quad k \in \{1,\ldots,n\},
\end{align*}
where the $A_k$ are \emph{monomial} matrices,
and constructs
\begin{align}
H &:= \sum_{k=1}^n A_k \otimes B_k,
\tag{H0}
\end{align}
such that
\begin{align}
H \in \{-1,1\}^{n b \times n b}
\quad
\text{and}
\quad
H H^T &= n b I_{(n b)},
\tag{H1}
\end{align}
i.e. $H$ is a Hada\-mard matrix of order $n b$.
The paper \cite{Leo14Constructions} focuses on a special case of the construction,
satisfying the conditions

\begin{align}
 A_j \ast A_k = 0 \quad (j \neq k)&, \quad \sum_{k=1}^n A_k \in \{-1,1\}^{n \times n},
\notag
\\
 A_k A_k^T &= I_{(n)},
\notag
\\
 A_j A_k^T + \lambda_{j,k} A_k A_j^T &= 0 \quad (j \neq k),
\label{constructions-4}
\\
 B_j B_k^T - \lambda_{j,k} B_k B_j^T &= 0 \quad (j \neq k),
\notag
\\
 \lambda_{j,k} &\in \{-1,1\},
\notag
\\
\sum_{k=1}^n  B_k B_k^T &= n b I_{(b)},
\notag
\end{align}
where $\ast$ is the Hada\-mard (element-by-element) matrix product.
(That is, $(M \ast N)_{i,j} := M_{i,j} N_{i,j}$ for all pair of matrices $M,N$ of the same shape.)

If, in addition, we stipulate that $A_j^2 = \kappa_j = \pm 1$ for $j$ from 1 to $n$,
we can now recognize that the $n$ matrices $A_1$ to $A_n$ are also the images,
under a real representation of order $n$,
of the generators of a real special quasi-Clifford algebra,
with $\lambda_{j,k}=\kappa_j \kappa_k (-1)^{1+\delta_{j,k}}$.
Thus $n$ must be a power of 2 large enough for this representation to exist, or a multiple of such a power.

In Section 3 of the paper  \cite{Leo14Constructions},
it is noted that the Clifford algebra $\R(2^m) \simeq \R_{m,m}$ has a canonical basis consisting of $4^m$ real monomial matrices
with the following properties:

Pairs of basis matrices either commute or anticommute.
Basis matrices are either symmetric or skew,
and so the basis matrices $A_j, A_k$ satisfy
\begin{align}
 A_k A_k^T &= I_{(2^m)},
\quad
 A_j A_k^T + \lambda_{j,k} A_k A_j^T = 0 \quad (j \neq k),
\quad
\lambda_{j,k} \in \{-1,1\}.
\label{A-property-1}
\end{align}

Additionally, for $n=2^m$, we can choose a transversal of $n$ canonical basis matrices that
satisfies conditions \eqref{constructions-4} on the $A$ matrices,
\begin{align}
 A_j \ast A_k = 0 \quad (j \neq k)&,
\quad
\sum_{k=1}^n A_k \in \{-1,1\}^{n \times n}.
\label{A-property-2}
\end{align}

\section{Special quasi-Clifford algebras applied to the plug-in constructions}\label{sec-SQC-Plug-in}

The properties of the real SQC algebras yield an alternate proof of Theorem 5 of \cite{Leo14Constructions},
and provide an answer to the question of whether the order of the $B$ matrices used in that proof can be improved \cite[Question 1]{Leo14Constructions}.
That theorem is restated here as a proposition.
\begin{Proposition}
\label{pr-construction-works}
\cite[Theorem 5]{Leo14Constructions}
If $n$ is a power of 2, the construction (H0) with conditions \eqref{constructions-4}
can always be completed, in the following sense.
If an $n$-tuple of $A$ matrices which produce a particular $\lambda$
is obtained by taking a transversal of canonical basis matrices of the Clifford algebra $\R_{m,m}$,
an $n$-tuple of $B$ matrices with a matching $\lambda$ can always be found.
\end{Proposition}

\begin{proof}

~

\begin{enumerate}
\item
For some sufficiently large order $b$, form an $n$-tuple $(D_1, \ldots, D_n)$ of \newline
$\{-1,0,1\}$ monomial matrices whose amicability / anti-amicability graph is the edge-colour complement of that of $(A_1,\ldots,A_n)$.
To be precise,
\begin{align*}
 D_j D_k^T - \lambda_{j,k} D_k D_j^T = 0 \quad (j \neq k),
\end{align*}
where $\lambda$ is given by \eqref{A-property-1}.
This can be done because $D_1, \ldots, D_n$ are the images of generators of some real SQC algebra $\mathcal{C}$,
and therefore $b$ can be taken to be the order of an irreducible real representation of $\mathcal{C}$,
which, by Corollary \ref{Gas82-3-8} is a power of 2.
\item
Since $b$ is a power of 2, we can find a Hada\-mard matrix $S$ of order $b$.
The Sylvester Hadamard matrix of order $b$ will do.
The $n$-tuple $(D_1 S, \ldots, D_n S)$ of $\{-1,1\}$ matrices of order $b$ has the same amicability / anti-amicability graph as that of $(D_1, \ldots, D_n)$.
\item
The $n$-tuple of Hada\-mard matrices
$(B_1,\ldots,B_n)$ $=$ $(D_1 S, \ldots, D_n S)$ of order $b$ satisfies conditions \eqref{constructions-4} on the $B$ matrices, and completes the construction (H0).
\end{enumerate}
\end{proof}

The theory of SQC algebras is described by Gasti\-neau-Hills \cite{Gas80,Gas82} with enough detail to enable a concrete construction
of the type given in the  proof of Proposition \ref{pr-construction-works} to be carried
out for any given pattern of amicability / anti-amicability $\lambda$, and any arbitrary assignment $\kappa$ of squares of generators.

For example, consider the cases where all of the $A$ matrices are pairwise amicable, that is $\lambda_{j,k}=-1$ for $j \neq k$.
We thus require an $n$-tuple of mutually anti-amicable $\{-1,0,1\}$ matrices $(D_1,\ldots,D_n)$.

Consider the generators $\beta_{-q},\ldots,\beta_{-1},\beta_{1},\ldots,\beta_{p}$ where $p+q=n$,
$\beta_j^2=\kappa_j$, with $\kappa_j = -1$ if $j<0$, $\kappa_j=1$ if $j>0$,
and
\begin{align*}
\beta_j \beta_k &+ \kappa_j \kappa_k \beta_k \beta_j = 0.
\end{align*}
Thus generators whose squares have the same sign anticommute, and generators whose squares have opposite signs commute.
For any real monomial representation $\rho$, we have
\begin{align*}
\rho(\beta_j)^T &= \kappa_j \rho(\beta_j),
\intertext{so that}
\rho(\beta_j) \rho(\beta_k)^T
&= \kappa_k \rho(\beta_j) \rho(\beta_k)
= -\kappa_j \rho(\beta_k) \rho(\beta_j)
= -\rho(\beta_k) \rho(\beta_j)^T.
\end{align*}
Thus any representation gives a set of mutually anti-amicable matrices.

We have split the set of $n$ generators into disjoint subsets of size $p$ and $q$, where the generators within each subset pairwise anti-commute,
and each pair of generators, where one is taken from each subset, commute.
The whole set of generators thus generates the algebra $\R_{p,0} \otimes \R_{0,q}$, whose faithful representations are given by Table \ref{Tensor}.

\footnotesize{}
\begin{table}[ht]
\begin{align*}
\begin{array}{c|ccccccccc}
  & q \rightarrow \\
p &          0  &          1     &          2  &          3     &          4  &          5     &          6  &          7     &          8  \\ \hline
0 &    \Real    &    \Complex    &    \Quat    &   {}^2\Quat    &    \Quat(2) &    \Complex(4) &    \Real(8) &{}^2\Real(8)    &   \Real(16)  \\
1 &{}^2\Real    &{}^2\Complex    &{}^2\Quat    &   {}^4\Quat    &{}^2\Quat(2) &{}^2\Complex(4) &{}^2\Real(8) &{}^4\Real(8)    &{}^2\Real(16) \\
2 &    \Real(2) &    \Complex(2) &    \Quat(2) &   {}^2\Quat(2) &    \Quat(4) &    \Complex(8) &    \Real(16)&{}^2\Real(16)   &    \Real(32) \\
3 & \Complex(2) &{}^2\Complex(2) & \Complex(4) &{}^2\Complex(4) & \Complex(8) &{}^2\Complex(8) & \Complex(16)&{}^2\Complex(16)& \Complex(32) \\
4 &    \Quat(2) &    \Complex(4) &    \Real(8) &   {}^2\Real(8) &    \Real(16)&    \Complex(16)&    \Quat(16)&   {}^2\Quat(16)&    \Quat(32) \\
5 &{}^2\Quat(2) &{}^2\Complex(4) &{}^2\Real(8) &   {}^4\Real(8) &{}^2\Real(16)&{}^2\Complex(16)&{}^2\Quat(16)&   {}^4\Quat(16)&{}^2\Quat(32) \\
6 &    \Quat(4) &    \Complex(8) &    \Real(16)&   {}^2\Real(16)&    \Real(32)&    \Complex(32)&    \Quat(32)&   {}^2\Quat(32)&    \Quat(64) \\
7 & \Complex(8) &{}^2\Complex(8) & \Complex(16)&{}^2\Complex(16)& \Complex(32)&{}^2\Complex(32)& \Complex(64)&{}^2\Complex(64)& \Complex(128) \\
8 &    \Real(16)&    \Complex(16)&    \Quat(16)&   {}^2\Quat(16)&    \Quat(32)&    \Complex(64)&    \Real(64)&  {}^2\Real(128)&    \Real(256)
\end{array}
\end{align*}
\caption{Tensor Products of real Clifford algebras $\R_{p,0} \otimes \R_{0,q}$.}
\label{Tensor}
\end{table}
\normalsize{}

\begin{table}[ht]
\begin{align*}
\begin{array}{|ccc|}
\hline
\text{Algebra} & \text{Faithful}       & \text{Irreducible}
\\
               & \text{Representation} & \text{Dimension}
\\
\hline
\R_{2,0} \otimes \R_{0,0} &\Real(2) & 2
\\
\R_{1,0} \otimes \R_{0,1} &{}^2\Complex & 2
\\
\R_{0,0} \otimes \R_{0,2} & \Quat & 4
\\
\hline
\end{array}
\end{align*}
\caption{Tensor Products of real Clifford algebras with $p+q=2$.}
\label{Tensor-p-q-2}
\end{table}

\begin{table}[ht]
\begin{align*}
\begin{array}{|ccc|}
\hline
\text{Algebra} & \text{Faithful}       & \text{Irreducible}
\\
               & \text{Representation} & \text{Dimension}
\\
\hline
\R_{4,0} \otimes \R_{0,0} & \Quat(2) & 8
\\
\R_{3,0} \otimes \R_{0,1} & {}^2\Complex(2) & 4
\\
\R_{2,0} \otimes \R_{0,2} & \Quat(2) & 8
\\
\R_{1,0} \otimes \R_{0,3} & {}^4\Quat & 4
\\
\R_{0,0} \otimes \R_{0,4} & \Quat(2) & 8
\\
\hline
\end{array}
\end{align*}
\caption{Tensor Products of real Clifford algebras with $p+q=4$.}
\label{Tensor-p-q-4}
\end{table}

\begin{table}[ht]
\begin{align*}
\begin{array}{|ccc|}
\hline
\text{Algebra} & \text{Faithful}       & \text{Irreducible}
\\
               & \text{Representation} & \text{Dimension}
\\
\hline
\R_{8,0} \otimes \R_{0,0} & \Real(16) & 16
\\
\R_{7,0} \otimes \R_{0,1} & {}^2\Complex(8) & 16
\\
\R_{6,0} \otimes \R_{0,2} & \Real(16) & 16
\\
\R_{5,0} \otimes \R_{0,3} & {}^4\Real(8) & \ \ 8
\\
\R_{4,0} \otimes \R_{0,4} & \Real(16) & 16
\\
\R_{3,0} \otimes \R_{0,5} & {}^2\Complex(8) & 16
\\
\R_{2,0} \otimes \R_{0,6} & \Real(16) & 16
\\
\R_{1,0} \otimes \R_{0,7} & {}^4\Real(8) & \ \ 8
\\
\R_{0,0} \otimes \R_{0,8} & \Real(16) & 16
\\
\hline
\end{array}
\end{align*}
\caption{Tensor Products of real Clifford algebras with $p+q=8$.}
\label{Tensor-p-q-8}
\end{table}

The relevant representations and the dimensions of the corresponding irreducible real monomial representations for $p+q=2,4$ and $8$
are given by Tables \ref{Tensor-p-q-2} to \ref{Tensor-p-q-8} respectively.
Due to the periodicity of 8 of real representations of real Clifford algebras,
in general, for $n=2^m$, for $m>2$, there exists a real special quasi-Clifford algebra with an
irreducible real monomial representation of order $2^{n/2-1}$ containing $n$ pairwise anti-amicable $\{-1,0,1\}$ matrices.

\paragraph*{Hurwitz-Radon theory.}
\label{sec-Hurwitz-Radon}

The following definition and proposition are taken from the author's recent paper on twin bent functions
and Hurwitz-Radon theory \cite{Leo17Hurwitz}.

A set of real orthogonal matrices $\{A_1,A_2,\ldots,A_s\}$ is called a \emph{Hurwitz-Radon family}
\cite{GerP74a,Hur22,Rad22} if
\begin{enumerate}
 \item
$A_j^T = -A_j$ for all $j=1,\ldots,s$, and
 \item
$A_j A_k = -A_k A_j$ for all $j \neq k$.
\end{enumerate}
The Hurwitz-Radon function $\rho$ is defined by
\begin{align*}
\rho(2^{4 d + c}) &:= 2^c + 8 d, \quad \text{where~} 0 \leqslant c < 4.
\end{align*}
As stated by Geramita and Pullman \cite[Theorem A]{GerP74a}, Radon
proved the following result \cite{Rad22}.
\begin{Proposition}\label{Hurwitz-Radon-lemma}
Any Hurwitz-Radon family of order $N$ has at most $\rho(N)-1$ members.
\end{Proposition}

As an immediate consequence of this proposition, at most $\rho(N)$ monomial
$\{-1,0,1\}$ matrices can be mutually anti-amicable.
The construction above in the case where all of the $A$ matrices are pairwise amicable, $n=2^m$ and $m>2$
corresponds to the case $c=3$, $d=2^{m-3}-1$, since
\begin{align*}
4 d + c &= 2^{m-1} + 4 - 1 = n/2 - 1,
\\
2^c + 8 d &= 8 + 2^m - 8 = n.
\end{align*}

\section{Discussion}

The construction used in the proof of Proposition~\ref{pr-construction-works} is a special case of
the construction (H0) with conditions \eqref{constructions-4}.
All of the low order cases investigated so far have been of this form.
This prompts two questions:
\begin{enumerate}
\item Are all instances of construction (H0) with conditions \eqref{constructions-4}
given by the special construction used in the proof of Proposition~\ref{pr-construction-works}?
\item Must the order of the $B$ matrices used in construction (H0) with conditions \eqref{constructions-4} always be a power of 2?
\end{enumerate}
Perhaps a deeper study of the representation theory of Gasti\-neau-Hills quasi-Clifford algebras could be used to address these questions.

~

\paragraph*{Acknowledgements.}

Thanks to Jennifer Seberry, who supervised the PhD thesis of Humphrey Gasti\-neau-Hills,
and brought his work to the attention of the author.
This work, including previous papers on this topic \cite{Leo14Constructions,Leo15Twin,Leo17Hurwitz}
began in 2007 while the author was a Visiting Fellow at the Australian National University;
continued while the author was a Visiting Fellow and a Casual Academic at the University of Newcastle, Australia;
and concluded while the author was an employee of the Australian Government in the Bureau of Meteorology, and also
an Honorary Fellow of the University of Melbourne.





\end{document}